\documentclass[11pt]{amsart}
\newcommand{\preprint}[1]{}

\newcommand{\hide}[1]{}

\usepackage{amssymb}
\usepackage{amsbsy}
\usepackage{amscd}
\usepackage{amsmath}
\usepackage{amsthm}

\input xy
\xyoption{all}

\numberwithin{equation}{section}

\theoremstyle{plain}
\newtheorem{thm}{Theorem}[section]

\newtheorem{cor}[thm]{Corollary}
\newtheorem{lem}[thm]{Lemma}
\newtheorem{assumption}[thm]{Assumption}

\theoremstyle{definition}
\newtheorem{defi}[thm]{Definition}

\newtheorem{rem}[thm]{Remark}

\newcommand{\K}{{\mathcal K}}

\newcommand{\M}{{\mathcal M}}
\newcommand{\fM}{{\mathfrak M}}
\newcommand{\tfM}{\widetilde{{\mathfrak M}}}

\newcommand{\PP}{{\mathbb P}}

\newcommand{\U}{{\mathcal U}}

\newcommand{\X}{{\mathcal X}}

\newcommand{\RealNumbers}{{\mathbb R}}
\newcommand{\Integers}{{\mathbb Z}}
\newcommand{\ComplexNumbers}{{\mathbb C}}
\newcommand{\RationalNumbers}{{\mathbb Q}}

\newcommand{\StructureSheaf}[1]{{\mathcal O}_{#1}}
\newcommand{\EndProof}{\hfill  $\Box$}

\newcommand{\rank}{{\rm rank}}

\newcommand{\Pic}{{\rm Pic}}
\newcommand{\Sym}{{\rm Sym}}
\newcommand{\Ext}{{\rm Ext}}

\newcommand{\SheafHom}{{\mathcal H}om}
\newcommand{\SheafEnd}{{\mathcal E}nd}

\newcommand{\SheafExt}{{\mathcal E}xt}

\newcommand{\Wedge}[1]{\wedge^{#1}}

\newcommand{\Choose}[2]{\left(\!\!\begin{array}{c}#1\\#2\end{array}\!\!\right)}

\begin{document}
\title[Naturality of the hyperholomorphic sheaf]
{Naturality of the hyperholomorphic sheaf over the cartesian square of a manifold of $K3^{[n]}$-type}
\author{Eyal Markman}
\address{Department of Mathematics and Statistics, 
University of Massachusetts, Amherst, MA 01003, USA}
\email{markman@math.umass.edu}

\date{\today}

\begin{abstract}
Let $\M$ be a $2n$-dimensional smooth and compact moduli space of stable sheaves on a $K3$ surface $S$ and 
$\U$ a universal sheaf over $S\times \M$. Over $\M\times\M$ there exists a natural reflexive sheaf $E$ of rank $2n-2$,
namely the first relative extension sheaf of the two pullbacks of $\U$ to $\M\times S\times \M$. We prove that $E$ is 
$\omega\boxplus \omega$-slope-stable with respect to every K\"{a}hler class $\omega$ on $\M$. The sheaf $E$ is known to deform to a sheaf $E'$ over $X\times X$, for every manifold $X$ deformation equivalent to $\M$, and  we prove that $E'$ is $\omega\boxplus \omega$-slope-stable with respect to every K\"{a}hler class $\omega$ on $X$. This triviality of the stability chamber structure combines with a 
result of S. Mehrotra and the author to show that the deformed sheaf $E'$ is canonical; each component of the 
the moduli space of marked triples $(X,\eta,E')$, where $\eta$ is a marking of $H^2(X,\Integers)$, maps isomorphically onto the component of the moduli space of marked pairs $(X,\eta)$ by forgetting $E'$. Consequently, the pretriangulated $K3$ category associated to the pair $(X\times X,E')$ in \cite{generalized-deformations} depends only on the isomorphism class of $X$.
\end{abstract}

\maketitle
\tableofcontents
%
\section{Introduction}
%
\subsection{Stability of modular sheaves}
An {\em irreducible holomorphic symplectic manifold} is a simply connected compact K\"{a}hler manifold, such that $H^0(X,\Omega^2_X)$ is one dimensional spanned by an everywhere non-degenerate holomorphic $2$-form. 
Every K\"{a}hler manifold $X$ deformation equivalent to the Hilbert scheme $S^{[n]}$ 
of length $n$ subschemes of a $K3$ surface $S$ is an irreducible holomorphic symplectic manifold \cite{beauville-varieties-with-zero-c-1}.
Such $X$ is said to be of {\em $K3^{[n]}$-type}. Every smooth and projective moduli space $\M$ of stable sheaves  on a $K3$ surface is 
an irreducible holomorphic symplectic manifold of $K3^{[n]}$-type,  by results of Huybrechts, Mukai, O'Grady, and Yoshioka \cite{ogrady-hodge-str,yoshioka-abelian-surface}.

Let $S$ be a $K3$ surface, $H$ an ample line bundle on $S$, $v$ a primitive Mukai vector, 
and $\M:=\M_H(v)$ a smooth and compact moduli space of $H$-stable sheaves on $S$ of Mukai vector $v$.
Assume that $2n:=\dim_\ComplexNumbers\M\geq 4.$ 
Denote by $\pi_S$ and $\pi_\M$ the two projections from $S\times \M$. 
Let $\U$ be a universal sheaf over $S\times \M$. 
There is a Brauer class $\theta$ in the cohomology group $H^2_{an}(\M,\StructureSheaf{\M}^*)$, with respect to the analytic topology of $\M$, such that $\U$ is $\pi_\M^*\theta$-twisted.
Let $\pi_{ij}$ be the projection from $\M\times S\times \M$ onto the product of the $i$-th and $j$-th factors. 
Let 
\begin{equation}
\label{eq-E}
E:= \SheafExt^1_{\pi_{13}}\left(\pi_{12}^*\U,\pi_{23}^*\U\right)
\end{equation}
be the relative extension sheaf over $\M\times \M$. 
Let $f_i$ be the projection from $\M\times\M$ onto the $i$-th factor. $E$ is a reflexive $f_1^*\theta^{-1}f_2^*\theta$-twisted sheaf of rank $2n-2$, which is locally free away from the diagonal, by \cite[Prop. 4.1]{markman-hodge}.
Given a K\"{a}hler class $\omega$ on $\M$, denote by
$\tilde{\omega}:=f_1^*\omega+f_2^*\omega$ the corresponding K\"{a}hler class over $\M\times \M$. 

\begin{defi}
\label{def-stability}
Let $X$ be a $d$-dimensional compact K\"{a}hler manifold and $\omega$ a K\"{a}hler class on $X$.
The {\em $\omega$-degree} of a coherent sheaf $G$ on $X$ is $\deg_\omega(G):=\int_X\omega^{d-1}c_1(G)$. 
If $G$ is torsion free of rank $r$,  its {\em $\omega$-slope} is $\mu_\omega(G):=\deg_\omega(G)/r$. 
Let $E$ be a torsion free coherent sheaf over $X$, 
which is $\theta$-twisted with respect to some Brauer class $\theta$. 
The sheaf $E$ is {\em $\omega$-slope-semistable}, if for every subsheaf $F$ of $E$, satisfying $0<\rank(F)<\rank(E)$, we have
\[
\deg_\omega\left(\SheafHom(E,F)\right)\leq 0.
\]
$E$ is said to be {\em $\omega$-slope-stable} if strict inequality holds above. $E$ is said to be 
{\em $\omega$-slope-polystable}, if it is $\omega$-slope-semistable as well as the direct sum of 
$\omega$-slope-stable sheaves. 
\end{defi}


\begin{thm}
\label{main-thm}
The sheaf $E$, given in Equation (\ref{eq-E}), is $\tilde{\omega}$-slope-stable with respect to every K\"{a}hler class $\omega$ on $\M$.
\end{thm}

The Theorem follows from the more general Theorem \ref{thm-stability} stated below. 
The Chern class $c_2(\SheafEnd(E))$, of the sheaf $E$ in
Theorem \ref{main-thm}, flatly deforms to a class of Hodge type $(2,2)$ on the cartesian square $X\times X$ of every manifold of $K3^{[n]}$-type \cite[Prop. 1.2]{markman-hodge}. 
This fact combines with Theorem \ref{main-thm} to imply that the sheaf 
$E$ is {\em $\tilde{\omega}$-hyperholomorphic} in the sense of Verbitsky \cite{kaledin-verbitski-book} (see also \cite[Cor. 6.11]{markman-hodge}). 
Verbitsky proved a very powerful deformation theoretic result for such sheaves \cite[Theorem 3.19]{kaledin-verbitski-book}.
Associated to a K\"{a}hler class $\omega$ on $\M$ is a {\em twistor family}
$\X\rightarrow \PP^1_\omega$ deforming $\M$ \cite{huybrects-basic-results}.
Verbitsky's theorem implies that $E$ extends to a reflexive sheaf over the fiber square $\X\times_{\PP^1_\omega}\X$ of the twistor family  associated to every K\"{a}hler class $\omega$ on $\M$.
Verbitsky's theorem, applied to the sheaf $E$ in Theorem \ref{main-thm},   
is central to our proof  with F. Charles of the Standard Conjectures for projective manifolds of $K3^{[n]}$-type
and to our work with S. Mehrotra on pretriangulated $K3$-categories associated to manifolds of $K3^{[n]}$-type
\cite{lefschetz,torelli,generalized-deformations}. 

The first example of a moduli space $\M$, for which the above Theorem holds, was given in \cite[Theorem 7.4]{markman-hodge}.
In that example the order of the Brauer class was shown to be equal to the rank of $E$, and so $E$ does not have any non-zero subsheaf of lower rank. Such a maximally twisted reflexive sheaf is thus $\tilde{\omega}$-slope-stable 
with respect to every K\"{a}hler class $\omega$ on $X$.

The special case of Theorem \ref{main-thm}, where $\M=S^{[n]}$ is the Hilbert scheme of length $n$ subschemes of a $K3$ surface $S$ 
with a trivial Picard group, was proven earlier in \cite[Theorem 1.1(1)]{markman-stability}. In that case it was proven that, though untwisted, 
$E$ again does not have any non-zero subsheaf of lower rank.

%
\subsection{Stability in terms of the singularities along the diagonal}
Let $X$ be an irreducible holomorphic symplectic manifold of complex dimension $2n>2$. 
Let $\beta:Y\rightarrow X\times X$ be the blow-up of the diagonal $\Delta\subset X\times X$. We get the diagram
\[
\xymatrix{
D \ar[r]^{\iota} \ar[d]_{p} & Y \ar[d]^{\beta}
\\
X \ar[r]_{\delta} & X\times X,
}
\]
where $\iota$  is the natural embedding, $\delta$ the diagonal embedding, and $p$ the natural projection. Note that $D$ is isomorphic to the projectivization of 
$TX$. Denote by $\ell\subset p^*TX$ the tautological line subbundle and let $\ell^\perp\subset p^*TX$ be its symplectic-orthogonal subbundle. Denote by $f_i$ the projection from $X\times X$ onto the $i$-th factor. 
Let $E$ be a $f_1^*\theta^{-1}f_2^*\theta$-twisted 
reflexive sheaf of rank $2n-2$ over $X\times X$, for some Brauer class $\theta$ on $X$,  which is locally free away from the diagonal.
Given a coherent sheaf $C$, denote by $C_{fr}$ the quotient of $C$ by its torsion subsheaf.
Set $V:=(\beta^*E)_{fr}(D)$. 

\begin{assumption}
\label{assumption-V}
Assume that $V$ is locally free 
and the restriction of $V$ to $D$  is isomorphic to $\ell^\perp/\ell$. 
\end{assumption}

\begin{thm}
\label{thm-stability}
The sheaf $E$ is $\tilde{\omega}$-slope-stable with respect to every K\"{a}hler class $\omega$ on $X$.
\end{thm}

The Theorem is proven in Section \ref{sec-proof-of-main-theorem}.

\begin{rem}
\label{rem-main-theorem-is-a-special-case}
Assumption \ref{assumption-V} holds for the sheaf $E$ in Equation (\ref{eq-E}), by 
\cite[Prop. 4.1]{markman-hodge}. Theorem \ref{main-thm} thus follows from the more general Theorem \ref{thm-stability}.
\end{rem}

%
\subsection{An isomorphism of two moduli spaces}

The second integral cohomology of an irreducible holomorphic symplectic manifold $X$ 
comes with a non-degenerate symmetric bilinear pairing  
known as the {\em Beauville-Bogomolov-Fujiki pairing}. Its signature is $(3,b_2(X)-3)$, 
where $b_2(X)$ is the second Betti number \cite{beauville-varieties-with-zero-c-1}. 
A {\em $\Lambda$-marking} for $X$ is an isometry $\eta:H^2(X,\Integers)\rightarrow \Lambda$ with a lattice $\Lambda$.
The moduli space ${\mathfrak M}_\Lambda$ of isomorphism classes of $\Lambda$-marked irreducible holomorphic symplectic manifolds $(X,\eta)$ is a non-Hausdorff manifold of dimension $\rank(\Lambda)-2$ \cite{huybrects-basic-results}. 

The pair $(\M\times \M,E)$  in Theorem \ref{main-thm} is known to deform to all cartesian squares 
$X\times X$ of manifolds of $K3^{[n]}$-type \cite[Theorem 1.3]{markman-hodge}. 
The sheaf $E$ is infinitesimally rigid, $\Ext^1(E,E)=0$, by \cite[Lemma 5.2]{generalized-deformations}. We describe next a global analogue of these two facts.

In \cite{torelli} we constructed a moduli space ${\widetilde{\mathfrak M}}_\Lambda$ of equivalence classes 
of triples $(X,\eta,E)$, where $X$ is of $K3^{[n]}$-type, $n\geq 2$, 
$\eta$ is a $\Lambda$-marking for $X$, and $E$ is a rank $2n-2$ reflexive infinitesimally rigid twisted sheaf over $X\times X$,
which is $\tilde{\omega}$-slope-stable\footnote{In \cite{torelli}
it was also assumed that the sheaf $\SheafEnd(E)$ is $\tilde{\omega}$-slope-polystable, but the latter property follows from 
the $\tilde{\omega}$-slope-stability of $E$, as was proven later in \cite[Prop. 6.5]{markman-hodge} for twisted sheaves (and is well known for untwisted sheaves).
}
with respect to some K\"{a}hler class $\omega$ on $X$. 
Two pairs $(X,\eta,E)$ and $(X',\eta',E')$ are {\em equivalent}, if there exists an isomorphism $f:X\rightarrow X'$, such that 
$\eta'=\eta\circ f^*$, and the sheaves $\SheafEnd(E)$ and $(f\times f)^*\SheafEnd(E')$ are isomorphic as sheaves of algebras.
The sheaf $E$ of every triple in ${\widetilde{\mathfrak M}}_\Lambda$ 
is assumed to satisfy \cite[Condition 1.6]{torelli}, which implies 
Assumption \ref{assumption-V} above. Fix a connected component ${\widetilde{\mathfrak M}}_\Lambda^0$ of 
${\widetilde{\mathfrak M}}_\Lambda$ and let 
\[
\phi:{\widetilde{\mathfrak M}}_\Lambda^0 \rightarrow {\mathfrak M}_\Lambda^0
\]
be the forgetful morphism sending a triple $(X,\eta,E)$ to the marked pair $(X,\eta)$, where 
${\mathfrak M}_\Lambda^0$ is the corresponding connected component of ${\mathfrak M}_\Lambda$.
The forgetful morphism $\phi$ is a surjective local homeomorphism, which is generically injective, by \cite[Theorems 1.9 and 6.1]{torelli}
(conditional on \cite[Conj. 1.12]{torelli}). 
Following is the main application of Theorem \ref{thm-stability}.

\begin{cor} 
\label{cor-torelli}
The above morphism $\phi$ is an isomorphism.
\end{cor}

The Corollary is proven in Section \ref{sec-proof-of-main-theorem} (conditional on \cite[Conj. 1.12]{torelli}). 
Taking the quotient by the monodromy action, Corollary \ref{cor-torelli} may be reformulated as follows.

\begin{cor}
\label{cor-torelli-without-marking}
On the cartesian square $X\times X$ of every manifold $X$ of $K3^{[n]}$-type there exists a {\em canonical} 
pair\footnote{When $n=2$, $E^*$ is isomorphic to $E\otimes \det(E)^*$ and the two sheaves of algebras 
$\SheafEnd(E)$ and $\SheafEnd(E^*)$ are isomorphic.} 
of isomorphism classes of sheaves of algebras $\SheafEnd(E)$ and $\SheafEnd(E^*)$, where $E$ is a reflexive rank $2n-2$ twisted sheaf, locally free away from the diagonal, which satisfies Assumption \ref{assumption-V} and is $\tilde{\omega}$-slope-stable with respect to every K\"{a}hler class $\omega$ on $X$. 
\end{cor}

The Corollary is proven in Section \ref{sec-proof-of-main-theorem} (conditional on \cite[Conj. 1.12]{torelli}). 
The transposition $\tau:X\times X\rightarrow X\times X$ of the two factors pulls back $\SheafEnd(E)$ to $\SheafEnd(E^*)$ \cite[Lemma 4.3]{generalized-deformations}.
The Corollary thus implies that the construction of \cite{generalized-deformations}, associating to a triple $(X,\eta,E)$ in $\tfM_\Lambda^0$ a
pretriangulated $K3$-category, results in a single equivalence class of such a category for each isomorphism class of a manifold $X$ of $K3^{[n]}$-type with non-maximal Picard rank. 

%
\section{Restriction to the blown-up diagonal}

Let $E$ be the sheaf in Theorem \ref{thm-stability} and 
let $F$ be a saturated subsheaf of  $E$ of rank in the range $0<\rank(F)<2n-2$. 
Let $\tilde{G}$ be the image of the composition $(\beta^*F)(D)\rightarrow (\beta^*E)(D)\rightarrow V$. 
Denote by $G$ the saturation of $\tilde{G}$ as a subsheaf of $V$. The sheaf $G$ is  reflexive, since $V$ is locally free. 
The restriction of $G$ to $D$ maps to the restriction  $\ell^\perp/\ell$ of $V$ and we denote by $G'$ the saturation of its image as a subsheaf of $\ell^\perp/\ell$.
We provide in this section an upper bound for $\deg_{\tilde{\omega}}(\SheafHom(E,F))$ in terms of the first Chern class of $G'$
(Lemma \ref{lemma-degree-of-Hom-E-F-in-terms-of-G-prime}). 
We then provide information on  $c_1(G')$ (Lemma \ref{lemma-c-1-G-prime-not-multiple-of-h}).

\begin{lem}
\label{lemma-c-1-of-Hom-is-push-forward-of-c-1-of-Hom}
The equality $c_1(\SheafHom(E,F))=\beta_*(c_1(\SheafHom(V,G)))$ holds. 
\end{lem}

\begin{proof}
The higher direct image $R^i\beta_*(\SheafHom(V,G))$, $i>0$, is supported on the diagonal and so its first Chern class vanishes.
The equality $c_1(\beta_*\SheafHom(V,G))=c_1(R\beta_*\SheafHom(V,G))$ follows. 
Similarly, $c_1(\beta_*\SheafHom(V,G))=c_1(\SheafHom(E,F))$, as the two sheaves coincide away from the diagonal.
Hence, $c_1(\SheafHom(E,F))$ is the graded summand in degree $2$ of
$\beta_*(ch(\SheafHom(V,G))td_\beta)$, by Grothendieck-Riemann-Roch \cite{grr-complex}. 
We have $c_1(TY)=(1-2n)[D]$ and so 
\[
td_\beta=td(TY)/td(T[X\times X])=1+\frac{1-2n}{2}[D]+\dots
\]
The statement follows, by the vanishing
of $\beta_*([D])$.
\end{proof}

Set $h:=c_1(\ell^{-1})\in H^2(D,\Integers)$.

\begin{lem}
\label{lemma-powers-of-h}
$
p_*(h^i)=\left\{
\begin{array}{ccl}
0 & \mbox{if} & i<2n-1 \ \mbox{or} \ i \ \mbox{is even},
\\
1 &  \mbox{if} & i=2n-1,
\end{array}
\right.
$
\\
and 
\begin{equation}
\label{eq-recursive-formula}
p_*(h^{2n+k})=-c_{k+1}(TX)-\sum_{\begin{array}{c}j=2\\ j\ \mbox{even}\end{array}}^{k-1}c_j(TX)p_*(h^{2n+k-j}), 
\end{equation}
for any positive odd integer $k$.
\end{lem}

\begin{proof}
The statement is well known. 
The vanishing of $p_*(h^i)$, for $i<2n-1$, follows for dimension reasons.
The equality $p_*(h^{2n-1})=1$ is proven in \cite[Appendix B.4 Lemma 9]{fulton}.
Consider the short exact sequence
$
0\rightarrow \ell\rightarrow p^*TX\rightarrow q\rightarrow 0
$
defining $q$. Then $c_{2n}(q)=0$, since $\rank(q)=2n-1$. On the other hand, the Chern polynomial of $q$ satisfies
\[
c_t(q)=p^*c_t(TX)/c_t(\ell)=p^*c_t(TX)(1+h+h^2 +\cdots + h^{4n-1}).
\]
Consequently,
$
0=c_{2n}(q)=\sum_{j=0}^{2n}p^*c_j(TX)h^{2n-j}
$ and
\[
h^{2n}=-\left[
\sum_{j=1}^{2n}p^*c_j(TX)h^{2n-j}
\right].
\]
The projection formula yields
\[
p_*(h^{2n+k})=
-p_*\left(\sum_{j=1}^{2n}p^*c_j(TX)h^{2n+k-j}\right)=
-c_{k+1}(TX)-
\sum_{j=1}^{k}c_j(TX)p_*(h^{2n+k-j}),
\]
for every positive integer $k$.
The vanishing of $p_*(h^i)$ for even $i$ follows, by induction, from the vanishing of $c_j(TX)$ for odd $j$.
The recursive formula (\ref{eq-recursive-formula}) follows. 
\end{proof}

\begin{lem}
\label{lemma-sym-TX}
The $t$-th symmetric power  $\Sym^{t}(TX)$ of the tangent bundle is $\omega$-slope-stable with respect to every K\"{a}hler class $\omega$ on $X$, for all $t\geq 0$.
The space $H^0(X,(\Wedge{j}TX)\otimes \Sym^t(T^*X))$ vanishes, for all $j\geq 0$ and all $t>1$.
\end{lem}

\begin{proof}
The vector bundle $TX$ admits a 
Hermite-Einstein metric whose $(1,1)$-form represents  $\omega$, for every K\"{a}hler class $\omega$ on $X$, by Yau's proof of the Calabi Conjecture 
\cite[Cor. 4.B.22]{huybrects-complex-geometry-book}. In particular, $TX$ is $\omega$-slope-polystable with respect to every K\"{a}hler class $\omega$.
The holonomy group of the tangent bundle of an irreducible holomorphic symplectic manifold of complex dimension $2n$ is the symplectic group $Sp(n)$ \cite[Prop. 4]{beauville-varieties-with-zero-c-1}. Consequently, the vector bundle $TX$ is slope-stable with respect to every K\"{a}hler class, its tensor powers are poly-stable, and the indecomposable direct summands of the tensor powers correspond to irreducible representations of $Sp(n)$. Let $U$ be the standard representation of $Sp(n)$. 
The symmetric powers $\Sym^t(U)$ are irreducible representations, for all $t\geq 0$. If $t>1$ then $\Sym^t(U)$ does not appear as a subrepresentation of the exterior product $\Wedge{j}U$, for any $j\geq 0$. Hence, 
$
H^0\left(X,\left(\Wedge{j}TX\right)\otimes \Sym^t(T^*X)\right)
$
vanishes, for $t>1$. 
\end{proof}

\begin{lem}
\label{lemma-effective-divisor-in-D}
Let $Z$ be a non-zero effective divisor on $D$ and $[Z]\in H^2(D,\Integers)$ its class. Then 
\[
\int_D p^*(\omega)^{2n-1}h^{2n-1}[Z]> 0,
\]
for every K\"{a}hler class $\omega$ on $X$.
\end{lem}

\begin{proof}
We have a direct sum decomposition $\Pic(D)=p^*\Pic(X)\oplus \Integers \ell$. 
The restriction of $\StructureSheaf{D}(Z)$ to each fiber of $p$ is effective. 
Hence, 
$\StructureSheaf{D}(Z)$ is isomorphic to $\ell^{-a}\otimes p^*L$, for some line bundle $L\in \Pic(X)$ and for some nonnegative integer $a$.
The space $H^0(D,\StructureSheaf{D}(Z))$ does not vanish and is isomorphic to
$H^0(X,L\otimes \Sym^{a}T^*X)$. Hence, $L^{-1}$ is a subsheaf of $\Sym^{a}T^*X$.
The vector bundle $\Sym^{a}T^*X$ is an $\omega$-slope-stable bundle with a trivial determinant, by Lemma 
\ref{lemma-sym-TX}. 
If $a>0$, then $L^{-1}$ is a subsheaf of lower rank. If $a=0$, then $L^{-1}$ is the ideal sheaf of a non-zero effective divisor on $X$, since $Z$ was assumed to be such. 
In both cases we get the inequality $\deg_\omega(L)>0$ and so  
\[
\int_D p^*(\omega)^{2n-1}h^{2n-1}[Z]=\int_X\omega^{2n-1}p_*(h^{2n-1}(ah+p^*c_1(L)))=
\int_X\omega^{2n-1}c_1(L)=\deg_\omega(L)>0,
\]
where the second equality is due to the vanishing of $p_*(h^{2n})$ and the equality $p_*(h^{2n-1})=1$ 
of Lemma \ref{lemma-powers-of-h}.
\end{proof}

Recall that the restriction of $V$ to $D$ is isomorphic to $\ell^\perp/\ell$, by Assumption \ref{assumption-V}.
Let $G'$ be the saturation of the image in $\ell^\perp/\ell$ of the restriction to $D$ of the subsheaf $G$ of $V$.
Note that the sheaves $\ell^\perp/\ell$ and $G'$ are untwisted.

\begin{lem}
\label{lemma-degree-of-Hom-E-F-in-terms-of-G-prime}
The following inequality holds for every K\"{a}hler class $\omega$ on $X$.
\[
\deg_{\tilde{\omega}}(\SheafHom(E,F))\leq
(2n-2)\Choose{4n-1}{2n}\left(\int_X\omega^{2n}\right)\int_D (p^*\omega)^{2n-1} h^{2n-1}c_1(G').
\]
\end{lem}

\begin{proof}
Let $a,b\in H^2(X,\Integers)$ be the classes satisfying $c_1(\SheafHom(E,F))=f_1^*a+f_2^*b$. We have
\begin{eqnarray}
\nonumber
\deg_{\tilde{\omega}}(\SheafHom(E,F))
&=&
\int_{X\times X}\tilde{\omega}^{4n-1}c_1(\SheafHom(E,F))
\\
\nonumber
&=&
\int_{X\times X}(f_1^*\omega+f_2^*\omega)^{4n-1}(f_1^*a+f_2^*b)
\\
\nonumber
&=&
\Choose{4n-1}{2n}\left(\int_X\omega^{2n}\right)\int_X\omega^{2n-1}(a+b)
\\
\nonumber
&=& \Choose{4n-1}{2n}\left(\int_X\omega^{2n}\right)\int_X\omega^{2n-1}\delta^*c_1(\SheafHom(E,F))
\\
\label{eq-degree-via-restriction-to-diagonal}
&=& \Choose{4n-1}{2n}\left(\int_X\omega^{2n}\right)\int_X\omega^{2n-1}\delta^*\beta_*c_1(\SheafHom(V,G)),
\end{eqnarray}
where the last equality is by Lemma \ref{lemma-c-1-of-Hom-is-push-forward-of-c-1-of-Hom}.
Lemma \ref{lemma-powers-of-h} yields 
\[
\delta^*\beta_*c_1(\SheafHom(V,G))=p_*\left[h^{2n-1}p^*\delta^*(\beta_*c_1(\SheafHom(V,G)))\right]=
p_*\left[h^{2n-1}\iota^*\beta^*(\beta_*c_1(\SheafHom(V,G)))\right].
\]
The equality
$\beta_*\beta^*\beta_*=\beta_*$ implies that the difference between the classes $\beta^*(\beta_*c_1(\SheafHom(V,G)))$
and $c_1(\SheafHom(V,G))$ belongs to the kernel of $\beta_*$ and is hence  a multiple of $[D]$. The vanishing 
$p_*(h^{2n-1}\iota^*[D])=-p_*(h^{2n})=0$ yields the equality
\[
\delta^*\beta_*c_1(\SheafHom(V,G))=p_*\left[h^{2n-1}\iota^*(c_1(\SheafHom(V,G)))\right].
\]

We have $\iota^*c_1(\SheafHom(V,G))=c_1(L\iota^*\SheafHom(V,G))=c_1([\ell^\perp/\ell]^*\otimes L\iota^*(G))$. 
The sheaf $G$ is reflexive, by construction, 
and hence its singular locus has codimension $\geq 3$ in $Y$. 
It follows that $c_1(L\iota^*G)=c_1(\iota^*G)=c_1((\iota^*G)_{fr})$. 
The subscheme of $Y$, where the rank of the homomorphism $G\rightarrow V$ is lower than the rank of $G$, has codimension at least $2$ in $Y$, since $G$ is a saturated subsheaf of $V$. Hence, 
the natural homomorphism $(\iota^*G)_{fr}\rightarrow G'$ is injective.
We conclude that $c_1(L\iota^*G)+[Z]=c_1(G')$, for some effective divisor $Z$ on $D$. 
The rank $2n-2$ vector bundle $\ell^\perp/\ell$ is symplectic, and hence $c_1(\ell^\perp/\ell)=0$. 
Hence
\[
\iota^*c_1(\SheafHom(V,G))=(2n-2)\left(c_1(G')-[Z]\right).
\]

The Projection Formula and the two displayed formulas above yield
\[
\int_X\omega^{2n-1}\delta^*\beta_*c_1(\SheafHom(V,G))=(2n-2)\int_D(p^*\omega)^{2n-1} h^{2n-1}\left(c_1(G')-[Z]\right).
\]
Lemma \ref{lemma-effective-divisor-in-D} yields the inequality
\[
\int_X\omega^{2n-1}\delta^*\beta_*c_1(\SheafHom(V,G))\leq (2n-2)\int_D(p^*\omega)^{2n-1} h^{2n-1}c_1(G').
\]
Lemma \ref{lemma-degree-of-Hom-E-F-in-terms-of-G-prime} follows from the above inequality and Equation
(\ref{eq-degree-via-restriction-to-diagonal}).
\end{proof}

\begin{lem}
\label{lemma-c-1-G-prime-not-multiple-of-h}
$c_1(G') = p^*\alpha- kh$, for a non-zero class $\alpha\in H^2(X,\Integers)$ and a positive integer $k$.
\end{lem}

\begin{proof}
There exist an integer $k$ and a class $\alpha\in H^2(X,\Integers)$ satisfying $c_1(G') = p^*\alpha- kh$, since
$H^2(D,\Integers)=p^*H^2(X,\Integers)\oplus\Integers h.$ The integer $k$ is positive, since the restriction of
$\ell^\perp/\ell$ to each fiber of $p:D\rightarrow X$ is slope-stable, by \cite[Lemma 7.4]{torelli}.

The rest of the proof is by contradiction. Assume that $\alpha=0$. Let $g$ be the rank of $G$. 
Then $0<g<2n-2$ and the top exterior power of $G'$ yields a line subbundle of $\Wedge{g}[\ell^\perp/\ell]$ isomorphic to 
$\ell^k$. In particular, the vector space
\begin{equation}
\label{eq-space-containing-global-section-associated-to-G-prime}
H^0(D,\Wedge{g}[\ell^\perp/\ell]\otimes \ell^{-k})
\end{equation}
contains a non-zero section.

We have the short exact sequence
\[
0\rightarrow \ell^\perp\rightarrow p^*TX\rightarrow \ell^{-1}\rightarrow 0.
\]
Dualizing, we get
\[
0\rightarrow \ell\rightarrow p^*T^*X\rightarrow (\ell^\perp)^*\rightarrow 0.
\]
Hence also the short exact 
\[
0\rightarrow \ell\otimes \Wedge{j-1}(\ell^\perp)^*\rightarrow p^*\Wedge{j}T^*X\rightarrow \Wedge{j}(\ell^\perp)^*\rightarrow 0.
\]
Dualizing the latter and tensoring by $\ell^{-t}$ we get the short exact sequence
\[
0\rightarrow \Wedge{j}(\ell^\perp)\otimes \ell^{-t}\rightarrow p^*\left(\Wedge{j}TX\right)\otimes \ell^{-t}\rightarrow 
\Wedge{j-1}(\ell^\perp)\otimes \ell^{-t-1}\rightarrow 0.
\]
The inclusion of $\Wedge{j-1}(\ell^\perp)\otimes \ell^{-t-1}$ in $p^*\left(\Wedge{j-1}TX\right)\otimes \ell^{-t-1}$ yields
the left exact sequence\footnote{Under the identification of $TX$ with $T^*X$ via the symplectic form, the rightmost homomorphism is the homomorphism  
$\left(\Wedge{j}TX\right)\otimes \Sym^t(TX) \rightarrow
\left(\Wedge{j-1}TX\right)\otimes \Sym^{t+1}(TX)$ 
appearing in the Koszul complex. We will not use this fact.}
\[
0\rightarrow p_*\left[\Wedge{j}(\ell^\perp)\otimes \ell^{-t}\right]
\rightarrow 
\left(\Wedge{j}TX\right)\otimes \Sym^t(T^*X)
\rightarrow
\left(\Wedge{j-1}TX\right)\otimes \Sym^{t+1}(T^*X).
\]

The vector space $H^0\left(X,\left(\Wedge{j}TX\right)\otimes \Sym^t(T^*X)\right)$ vanishes, for $t>1$, by Lemma \ref{lemma-sym-TX}.
We conclude that $H^0\left(D,\Wedge{j}(\ell^\perp)\otimes\ell^{-t}\right)$
vanishes, for all pairs $(j,t)$ of integers satisfying $j\geq 0$ and $t>1$.
The short exact sequence 
$0\rightarrow \ell\rightarrow \ell^\perp\rightarrow [\ell^\perp/\ell]\rightarrow 0$
yields 
\[
0\rightarrow \ell\otimes\Wedge{j-1}[\ell^\perp/\ell]\rightarrow \Wedge{j}\ell^\perp\rightarrow \Wedge{j}[\ell^\perp/\ell]\rightarrow 0.
\]
Tensoring by $\ell^{-t}$ we get
\begin{equation}
\label{eq-short-exact-with-wedge-ell-perp-mod-ell}
0\rightarrow \Wedge{j-1}[\ell^\perp/\ell]\otimes\ell^{1-t}\rightarrow \left(\Wedge{j}\ell^\perp\right)\otimes\ell^{-t}\rightarrow 
\Wedge{j}[\ell^\perp/\ell]\otimes\ell^{-t}\rightarrow 0.
\end{equation}
We conclude that $H^0\left(D,\Wedge{j-1}[\ell^\perp/\ell]\otimes\ell^{1-t}\right)$ vanishes, for all pairs $(j,t)$ of integers satisfying $j\geq 1$ and $t>1$. The vanishing of the space
(\ref{eq-space-containing-global-section-associated-to-G-prime}) follows, by taking $j=g+1$ and $t=k+1$.
This provides the desired contradiction.
\hide{
It remains to prove the vanishing of the space
(\ref{eq-space-containing-global-section-associated-to-G-prime}) for $k=1$. In this case 
$H^0(D,\Wedge{g}(\ell^\perp)\otimes \ell^{-1})$
maps isomorphically onto the kernel 
$H^0\left(X,\Wedge{g+1}T^*X\right)$ 
of
\[
H^0\left(X,\left(\Wedge{g}T^*X\right)\otimes T^*X\right)\rightarrow 
H^0\left(X,\left(\Wedge{g-1}T^*X\right)\otimes \Sym^2(T^*X)\right).
\]
It follows that $g$ is odd and 
$H^0\left(D,\Wedge{g}(\ell^\perp)\otimes \ell^{-1}\right)$ is one-dimensional.

Consider the case $k=1$ of (\ref{eq-short-exact-with-wedge-ell-perp-mod-ell})
\[
0\rightarrow \Wedge{g-1}[\ell^\perp/\ell]\rightarrow \left(\Wedge{g}\ell^\perp\right)\otimes\ell^{-1}\rightarrow 
\Wedge{g}[\ell^\perp/\ell]\otimes\ell^{-1}\rightarrow 0.
\]
Let $\bar{\sigma}$ be the global section of $H^0\left(D,\Wedge{2}[\ell^\perp/\ell]\right)$ induced by $p^*\sigma$. Then
$H^0\left(D,\Wedge{g-1}[\ell^\perp/\ell]\right)$ contains $\bar{\sigma}^{(g-1)/2}$ and it injects into the one dimensional space
$H^0\left(D,\Wedge{g}(\ell^\perp)\otimes \ell^{-1}\right)$. 
Considering the long exact sequence of sheaf cohomologies associated to the above sequence 
we conclude that the connecting homomorphism
\[
H^0\left(D,\Wedge{g}[\ell^\perp/\ell]\otimes\ell^{-1}\right)
\rightarrow H^1\left(D,\Wedge{g-1}[\ell^\perp/\ell]\right)
\]
is injective.
}
\hide{
\underline{Case $g>1$:}
The contraction homomorphism
\[
\rfloor\sigma^{\frac{g+1}{2}}:\Wedge{g}TX\rightarrow T^*X
\]
is surjective. The pullback of the latter restricts to a surjective homomorphism
\[
p^*\left(\Wedge{g-1}TX\right)\otimes \ell \rightarrow \ell^\perp,
\]
which in turn restricts to a homomorphism
\[
\left(\Wedge{g-1}\ell^\perp\right)\otimes \ell \rightarrow \ell^\perp.
\]
We claim that the latter homomorphism has rank $1$ and image $\ell$.

The latter  is the pullback of a  homomorphism
\[
\left(\Wedge{g-1}[\ell^\perp/\ell]\right)\otimes \ell \rightarrow \ell^\perp,
\]
}
\end{proof}

%
\section{Proof of Theorem \ref{thm-stability}}
\label{sec-proof-of-main-theorem}
The proof of Theorem \ref{thm-stability} reduces to that of the inequality 
\[
\int_D(p^*\omega)^{2n-1}h^{2n-1}c_1(G')<0,
\] 
for every K\"{a}hler class $\omega$ and for every non-zero saturated proper subsheaf $G'$ of $\ell^\perp/\ell$,
by Lemma \ref{lemma-degree-of-Hom-E-F-in-terms-of-G-prime}. 
Set $c_1(G')=p^*\alpha-kh$ as in Lemma \ref{lemma-c-1-G-prime-not-multiple-of-h}.
The left hand side of the above inequality is equal to 
$
\int_X \omega^{2n-1}\alpha,
$
by the projection formula and the vanishing of $p_*(h^{2n})$.

Given a positive integer $N$ 
we have the sequence of inclusions
\[
p_*(G'\otimes\ell^{-N})
\subset
p_*([\ell^\perp/\ell]\otimes\ell^{-N})
\subset
p_*\left(\Wedge{2}(\ell^\perp)\otimes \ell^{-N-1}\right)
\subset
p_*\left(\Wedge{2}(p^*TX)\otimes \ell^{-N-1}\right),
\]
where the second inclusion follows from the sequence (\ref{eq-short-exact-with-wedge-ell-perp-mod-ell}).
The projection formula yields the inclusion
$p_*(G'\otimes\ell^{-N})
\subset \left(\Wedge{2}TX\right)\otimes \Sym^{N+1}T^*X$. The latter is an $\omega$-polystable vector bundle with zero first Chern class. We conclude the inequality
\begin{equation}
\label{eq-degree-inequality-as-polynomial-in-N}
\int_X\omega^{2n-1}c_1(p_*(G'\otimes\ell^{-N}))\leq 0,
\end{equation}
for every K\"{a}hler class $\omega$ on $X$.

The higher direct images $R^ip_*(G'\otimes\ell^{-N})$ vanish, for all $N$ sufficiently large. Hence, there exists an integer $N_0$, such that 
\[
c_1(p_*(G'\otimes\ell^{-N}))=c_1(Rp_*(G'\otimes\ell^{-N})),
\]
for all $N>N_0$. The right hand side is the graded summand $P(N)$ in $H^2(X,\Integers)[N]$ of the polynomial
\[
p_*\left(ch(G')\exp(Nh)td_p\right) \in H^*(X,\RationalNumbers)[N]
\]
in the variable $N$ with coefficients in the cohomology ring.
Let $g$ be the rank of $G'$.
We have $c_1(T_p)=2nh$, 
\begin{eqnarray*}
td_p&=&1+nh+\dots,
\\ 
ch(G')&=&g+(p^*\alpha-kh)+\dots
\\
exp(Nh)&=& \sum_{j=0}^{4n-1}h^j\frac{N^j}{j!}.
\end{eqnarray*}
Hence, 
the graded summand of $ch(G')\exp(Nh)td_p$ in $H^{4n}(D,\RationalNumbers)$ is a polynomial 
of degree $2n$ in $N$ whose first two leading terms are
\[
gh^{2n}\frac{N^{2n}}{2n!}+
[(p^*\alpha-kh)+gnh]h^{2n-1}\frac{N^{2n-1}}{(2n-1)!}+\dots
\]
The polynomial $P(N)$ is the image of the above polynomial under $p_*$.
The vanishing of $p_*(h^{2n})$ yields that $P(N)$ has degree $2n-1$ in $N$ and its leading term 
is
\[
P(N):=c_1(Rp_*(G'\otimes\ell^{-N}))=\frac{N^{2n-1}}{(2n-1)!}\alpha+\dots
\]
The values of the left hand side of the inequality (\ref{eq-degree-inequality-as-polynomial-in-N}) for $N>N_0$ are thus the values of a 
polynomial in $\RealNumbers[N]$ of degree $2n-1$, whose leading coefficient is 
\[
\frac{1}{(2n-1)!}\int_X\omega^{2n-1}\alpha. 
\]
The  inequality (\ref{eq-degree-inequality-as-polynomial-in-N}) thus implies the inequality 
\begin{equation}
\label{eq-not-yet-strict-inequality}
\int_X\omega^{2n-1}\alpha\leq 0,
\end{equation} 
for every K\"{a}hler class $\omega$. 

Let $\K_X$ be the K\"{a}hler cone of $X$. The map 
$\K_X\rightarrow H^{2n-1,2n-1}(X,\RealNumbers)$ sending $\omega$ to $\omega^{2n-1}$ is an open map. Indeed,
its differential at $\omega$ is the cup product  
\[
(2n-1)\omega^{2n-2}\cup:H^{1,1}(X,\RealNumbers)\rightarrow H^{2n-1,2n-1}(X,\RealNumbers),
\]
which is an isomorphism by the Hard Lefschetz Theorem. 
The class $\alpha$ does not vanish, by Lemma \ref{lemma-c-1-G-prime-not-multiple-of-h}.
Hence, the linear functional 
\[
H^{2n-1,2n-1}(X,\RealNumbers)\rightarrow \RealNumbers,
\] 
sending $\lambda$ to
$\int_X\lambda\alpha$, is an open map as well. We conclude that 
the image of $\K_X$ under the composition of these two open maps is an open subset and 
the inequality (\ref{eq-not-yet-strict-inequality}) is strict. This completes the proof of Theorem \ref{thm-stability}.
\EndProof

\begin{proof}[{\bf Proof of Corollary \ref{cor-torelli}}]
Any two points $(X,\eta,E_1)$ and $(X,\eta,E_2)$ in the fiber of $\phi$ are inseparable, by 
\cite[Theorem 6.1]{torelli}. 
If, furthermore,  $E_1$ and $E_2$ are $\tilde{\omega}$-slope-stable, with respect to the same K\"{a}hler class $\omega$ on $X$, 
then $\SheafEnd(E_1)$ and $\SheafEnd(E_2)$ are isomorphic as sheaves of algebras, by \cite[Lemma 5.3]{torelli}.
The sheaves $E_i$, $i=1,2$, are assumed to satisfy \cite[Condition 1.6]{torelli}, which implies Assumption \ref{assumption-V}, by \cite[Lemma 7.6]{torelli}. Hence, the sheaves $E_i$, $i=1,2$, 
are $\tilde{\omega}$-slope-stable with respect to every K\"{a}hler class $\omega$ on $X$, by Theorem \ref{thm-stability}.
It follows that $\SheafEnd(E_1)$ and $\SheafEnd(E_2)$ are isomorphic as sheaves of algebras, and the two triples 
$(X,\eta,E_i)$, $i=1,2$, represent the same equivalence class in $\widetilde{{\mathfrak M}}^0_\Lambda$. 
\end{proof}

\begin{proof}[{\bf Proof of Corollary \ref{cor-torelli-without-marking}}]
The isometry group $O(\Lambda)$ of $\Lambda$ acts on $\fM_\Lambda$ by $g(X,\eta)\mapsto (X,g\eta)$.
Let $Mon(\Lambda)\subset O(\Lambda)$ be the subgroup leaving the connected component $\fM^0_\Lambda$ invariant.
$Mon(\Lambda)$ is determined in \cite[Theorem 1.2]{integral-constraints}, which implies, in particular, that elements of $Mon(\Lambda)$
act on the discriminant group $\Lambda^*/\Lambda\cong\Integers/(2n-2)\Integers$ as $\pm 1$.
Denote by 
\[
cov:Mon(\Lambda)\rightarrow \{\pm 1\}
\] 
the corresponding character.
The set of $Mon(\Lambda)$-orbits in $\fM_\Lambda^0$ is in bijection with the set of isomorphism classes of manifolds of $K3^{[n]}$-type. Conjugating the $Mon(\Lambda)$-action via the isomorphism $\phi$ of Corollary \ref{cor-torelli} lifts it to an action on
$\tfM^0_\Lambda$. The latter action is given by $g(X,\eta,E)\mapsto (X,g\eta,E^{cov(g)})$, where 
$E^{cov(g)}=E$, if $cov(g)=1$, and $E^{cov(g)}=E^*$, if $cov(g)=\nolinebreak -1$, by the proof of \cite[Theorem 1.11]{torelli}. 
The inverse image under $\phi$ of the $Mon(\Lambda)$ orbit of $(X,\eta_0)$ in $\fM^0_\Lambda$ consists of equivalence classes of triples $(X,\eta,E)$ in a unique $Mon(\Lambda)$ orbit  of, say, $(X,\eta_0,E_0)$, by Corollary \ref{cor-torelli}. 
The above description of the monodromy action on
$\tfM^0_\Lambda$ implies that $\SheafEnd(E)$ is isomorphic as a sheaf of algebras to $\SheafEnd(E_0)$ or $\SheafEnd(E_0^*)$.
\end{proof}

{\bf Acknowledgements:}
I would like to thank Sukhendu Mehrotra and Misha Verbitsky for stimulating conversations. I thank  Sukhendu Mehrotra for his careful reading of an early draft and for his helpful comments.


\end{document}